\documentclass[a4paper]{article}
\usepackage{amsmath,amsthm,amssymb}

\newtheorem{proposition}{Proposition}[section]

\newtheorem{theorem}[proposition]{Theorem}

\newtheorem{lemma}[proposition]{Lemma}

\newtheorem*{noCounterDef}{Definition}
\newtheorem*{noCounterPro}{Problem}

\newcommand{\noproof}{\unskip\nobreak\hfill\penalty50\hskip2em\hbox{}\nobreak\hfill%
       $\square$\parfillskip=0pt\finalhyphendemerits=0\par}

\newenvironment{txteq}
  {
    \begin{equation}
    \begin{minipage}[c]{0.9\textwidth} 
    \em                                
  }
  {\end{minipage}\ignorespacesafterend\end{equation}}

\newcommand{\emtext}[1]{\text{\em #1}}
\newcommand{\offen}[1]{\mathaccent"7017\relax #1}
\providecommand{\eps}{\varepsilon}
\newcommand{\abs}[1]{\left\vert #1\right\vert}
\newcommand{\assign}{\mathrel{\mathop{:}}=}
\newcommand{\rand}[1]{\partial #1}
\newcommand{\quer}[1]{\overline{#1}}
\providecommand{\RRR}{\mathcal{R}}
\providecommand{\NN}{\mathbb{N}}

\title{\textbf{End spaces of graphs are normal}}
\author{Philipp Spr\"ussel}
\date{\empty}

\begin{document}

\maketitle

\begin{abstract}
  \noindent
  We show that the topological space of any infinite graph and its ends is normal. In particular, end spaces themselves are normal.
\end{abstract}

\section{Introduction}

The notion of the \emph{ends} of an arbitrary infinite graph, not necessarily locally finite, was introduced by Halin~\cite{Hal64}. Jung~\cite{Jung71} defined a topology on the set of ends, which was extensively studied and extended to the set of vertices and ends by Polat~\cite{Pol90},~\cite{Pol96a},~\cite{Pol96b}. Diestel and K\"uhn~\cite{DieKue04a},~\cite{DieKue04b} extended this topology to the entire graph (vertices, edges and ends); see also~\cite{DieGT}. Some fundamental topological questions about this space---in particular, when it is compact or metrizable---were recently answered by Diestel~\cite{Die06}.

One basic question that has remained open, for the end space $\Omega(G)$ of a graph $G$ as well as for the space $\hat V(G)=V(G)\cup\Omega(G)$ and for the entire space $|G|$ including both $G$ and its ends, is whether or not this space is normal. (They are easily seen to be regular.) When $G$ is connected and locally finite, then $\abs{G}$ coincides with the Freudenthal compactification of the cell-complex corresponding to $G$, and hence both $\abs{G}$ and the closed subspace $\Omega(G)$ of its ends are normal. When $G$ has a normal spanning tree, then $\hat V(G)$ is metrizable~\cite{Pol96a} and so is $|G|$~\cite{Die06}, so all three spaces are normal. In this paper, we show that $\Omega(G)$, $\hat V(G)$, and $|G|$ are always normal.

\section{Notation, background, and statement of results}
\label{sec:notation}

We assume familarity with the basic notions of infinite graph theory, for example as presented in Diestel~\cite{DieGT}. However, let us quickly review some of these. One-way infinite paths are called \emph{rays}. Every subray of a ray $R$ is a \emph{tail of $R$}. The union of two rays that have a common starting vertex and are otherwise disjoint, is called a \emph{double ray}. We call two rays \emph{equivalent} if, for every finite set $S$ of vertices, both rays have a tail in the same component of $G-S$. It is quite easy to see that this is an equivalence relation, the equivalence classes are called \emph{ends}. The set of ends of $G$ is denoted by $\Omega(G)$.

The topology on $\Omega(G)$ is defined as follows. Let $S$ be a finite set of vertices and $C$ a component of $G-S$. Let $\Omega_S(C)$ denote the set of those ends of $G$ whose rays have a tail in $C$. The sets $\Omega_S(C)$, taken over all $S$ and $C$, form a basis of the topology on $\Omega(G)$. We call $\Omega(G)$ together with this topology the \emph{end space} of $G$.

\begin{noCounterDef}
  A Hausdorff space $X$ is called \emph{normal} if any two disjoint closed subsets have disjoint neighbourhoods.
\end{noCounterDef}
It is well known that compact Hausdorff spaces are normal, and so are metric spaces. But the end space of an arbitrary graph need be neither compact nor metrizable (see~\cite{Die06},~\cite{DieLea01} for characterizations of those that are), nor even have a countable basis. The standard ways to prove normality therefore fail; our proof will be from first principles.

The first main result of this paper is
\begin{theorem}
  \label{thm:ends}
  Let $G$ be an infinite graph. Then its end space $\Omega(G)$ is normal.
\end{theorem}

The topology on $\Omega(G)$ has been extended to $\hat V(G)\assign V(G)\cup\Omega(G)$ by Polat~\cite{Pol96a} and to the space $|G|$---which additionally contains the edges of $G$---by Diestel and K\"uhn~\cite{DieKue04a},~\cite{DieKue04b}.

The topology on $\hat V(G)$ is defined as follows. We start with the discrete topology on $V(G)$,~i.e.~for each vertex $v$, the set $\{v\}$ is open. For an end $\omega\in\Omega(G)$ and a finite set $S$ of vertices, denote by $C(S,\omega)$ the component of $G-S$ that contains a ray from $\omega$. The sets $\hat C(S,\omega)\assign C(S,\omega)\cup\Omega_S\big(C(S,\omega)\big)$, taken over all $S$, form a neighbourhood basis of $\omega$.

It is easy to see that the subspace topology $\hat V(G)$ induces on its closed subspace $\Omega(G)$ is exactly the topology on $\Omega(G)$ defined earlier.

As closed subspaces of normal spaces are also normal (note that this is not true for arbitrary subspaces), Theorem~\ref{thm:ends} therefore follows at once from our second main result:
\begin{theorem}
  \label{thm:Vhat}
  Let $G$ be an infinite graph. Then the space $\hat V(G)$ consisting of the vertices and ends of $G$ is normal.
\end{theorem}

We will prove Theorem~\ref{thm:Vhat} in Section~\ref{sec:proof}. In Section~\ref{sec:incl_edges}, we will discuss the space $|G|$ and the topologies given to it.


A spanning tree $T$ of $G$ is \emph{normal} if there is a vertex $r$ (called \emph{root} of $T$) such that for every edge $e$ of $G$, one of the endvertices of $e$ lies on the path in $T$ from $r$ to the other endvertex of $e$. If $G$ has a normal spanning tree, $\hat V(G)$ is metrizable~\cite[Theorems~5.8~\&~5.15]{Pol96a} and therefore Theorem~\ref{thm:Vhat} is trivial.

A set $S\subset V(G)$ \emph{separates} two points $x,y\in\hat V(G)$ if they do not lie (or have rays) in the same component of $G-S$. It separates two sets $A,B\subset\hat V(G)$ if it separates every point in $A$ from every point in $B$. If an end $\omega$ cannot be separated by finitely many vertices from a given (infinite) set $Z$ of vertices, then no ray $R$ in $\omega$ can be separated by finitely many vertices from $Z$. Thus, there are infinitely many disjoint paths from $V(R)$ to $Z$; the union of $R$ with these paths is called a \emph{comb}. The last vertices of these paths are called the \emph{teeth} of the comb, and $R$ is its \emph{spine}. A \emph{tail} of a comb is the union of a tail of its spine and all the paths that meet this tail. Note that not every vertex of the spine has to be the first vertex of one of the paths, and a tooth may lie on the spine if (and only if) its finite path is trivial. (See \cite{DieGT} for more on combs.)

We thus have the following lemma.
\begin{lemma}
  \label{lemma:comb}
  If an end $\omega$ of an infinite graph $G$ cannot be separated by finitely many vertices from a given set $Z$ of vertices, then there is a comb with teeth in $Z$ and spine in $\omega$.
\end{lemma}

Given a subset $U$ of a topological space $X$, we call the set of all points $x\in X$ such that every neighbourhood of $x$ meets both $U$ and $X\setminus U$ the \emph{boundary} of $U$ and denote it by $\rand U$. Further, we denote the \emph{closure} $U\cup\rand U$ of $U$ by $\quer{U}$.

We shall later need the following lemma.
\begin{lemma}
  \label{lemma:border}
  If $U$ is the union of (arbitrarily many) open sets $\hat C(S_i,\omega_i)$ in $\hat V(G)$, then $\rand U$ is contained in the closure of $\bigcup S_i$ in $\hat V(G)$.
\end{lemma}
\begin{proof}
  Assume, for contradiction, that $\rand U\setminus\quer{\bigcup S_i}\not=\emptyset$. First note that no vertex $v$ lies in $\rand U$, since the neighbourhood $\{v\}$ of $v$ avoids either $U$ or $\hat V(G)\setminus U$. Hence there is an end $\omega$ in $\rand U\setminus\quer{\bigcup S_i}$. There is a neighbourhood $\hat C(S,\omega)$ of $\omega$ that avoids $\bigcup S_i$. Thus for each $i$ the component $C(S,\omega)$ of $G-S$ is contained in some component of $G-S_i$ and hence either $\hat C(S,\omega)\subset\hat C(S_i,\omega_i)$ or $\hat C(S,\omega)\cap\hat C(S_i,\omega_i)=\emptyset$. If the latter holds for all $i$, we have $\hat C(S,\omega)\cap U=\emptyset$, contradicting the fact that $\omega\in\rand U$. On the other hand, if $\hat C(S,\omega)\subset\hat C(S_i,\omega_i)$ for at least one $i$, we have $\hat C(S,\omega)\subset U$, again a contradiction.
\end{proof}

\section{\large Proof of the normality theorem}
\label{sec:proof}

As we observed in Section~\ref{sec:notation}, Theorem~\ref{thm:Vhat} is trivial for graphs that have a normal spanning tree. For arbitrary graphs, Theorem~\ref{thm:Vhat} will follow easily from
\begin{lemma}
  \label{lemma:separatingends}
  Let $G$ be an infinite graph and $A,B\subset\Omega(G)$ disjoint closed sets in $\hat V(G)$. Then there exist disjoint neighbourhoods of $A$ and $B$ in $\hat V(G)$.
\end{lemma}

\begin{proof}
%
  If $A$ and $B$ are both countable and infinite (the case where $A$ or $B$ is finite is trivial), there is a simple way of constructing disjoint neighbourhoods of $A$ and $B$: Enumerate the ends in $A$ by $\omega_0,\omega_1,\dots$ and the ends in $B$ by $\tilde\omega_0,\tilde\omega_1,\dots$. Now for $i=0,1,\dots$, there are finite sets $S_i$ and $\tilde S_i$ of vertices separating $\omega_i$ from $B$ and $\tilde \omega_i$ from $A$, respectively. We now have neighbourhoods $U\assign\bigcup_{i<\omega}\hat C(S_i,\omega_i)$ of $A$ and $\tilde U\assign\bigcup_{j<\omega}\hat C(\tilde S_j,\tilde\omega_j)$ of $B$. These will be disjoint if the neighbourhoods $\hat C(S_i,\omega_i)$ and $\hat C(\tilde S_j,\tilde \omega_j)$ are disjoint for any $i,j$. To achieve this, it suffices to choose the separators $S_i$ in a special way, namely, containing $\bigcup_{j<i}\tilde S_j$. Then, for every $j<i$, $C(S_i,\omega_i)$ will be contained in the component $C(\tilde S_j,\omega_i)$ of $G-\tilde S_j$, which cannot be $C(\tilde S_j,\tilde \omega_j)$, because $C(\tilde S_j,\tilde \omega_j)$ avoids $A$ but $C(S_i,\omega_i)$ contains $\omega_i\in A$. Hence $\hat C(S_i,\omega_i)\cap\hat C(\tilde S_j,\tilde \omega_j)=\emptyset$. Likewise, we choose each $\tilde S_j$ so as to contain $\bigcup_{i\le j}S_i$, which ensures that $\hat C(\tilde S_j,\tilde \omega_j)$ will be disjoint from every $\hat C(S_i,\omega_i)$ with $i\le j$. Thus, $U$ and $\tilde U$ will be disjoint.
  
  This procedure fails for uncountable $A$ or $B$, as it may be impossible at a transfinite step for a finite separator $S_i$ to contain every previous separator.
  
  For $A$ and $B$ that are not necessarily countable, we shall construct a neighbourhood $U$ of $A$ in $\hat V(G)$, whose closure in $\hat V(G)$ will not meet $B$. The desired neighbourhood of $B$ can then be chosen as $\hat V(G)\setminus\quer{U}$, completing the proof of Lemma~\ref{lemma:separatingends}.
  
  Let us write $A=\{\omega_i\:|\: i<\lambda\}$. At step $i<\lambda$ we will choose a finite set $S_i$ of vertices separating $\omega_i$ from $B$ and put $U_i\assign\hat C(S_i,\omega_i)$. Finally, let $U\assign\bigcup_{i<\lambda}U_i$.
  
  Obviously, we are not allowed to choose the sets $S_i$ arbitrarily; the choice has to guarantee that $\quer{U}$ does not meet $B$. To find out how we may ensure that, let us take a look at what happens if we have chosen the sets $S_i$ already, but badly: there is an end $\omega\in B$ in $\quer{U}$. By choice of the $U_i$, we have $\omega\in\rand U$. Hence Lemma~\ref{lemma:border} yields $\omega\in\quer{\bigcup_{i<\lambda}S_i}$. Thus, $\omega$ cannot be separated from $\bigcup_{i<\lambda}S_i$ by finitely many vertices; hence Lemma~\ref{lemma:comb} yields a comb with spine in $\omega$ and teeth in $\bigcup_{i<\lambda}S_i$. As every $S_i$ is finite, the comb has teeth in infinitely many $S_i$. Our aim will be to choose the $S_i$ so that infinitely many of these teeth can be linked by disjoint rays to (pairwise different) ends in $A$. Then $\omega\in B$ will lie in the closure of these ends, and hence in $A$, contrary to our assumption that $A\cap B=\emptyset$.
  
  For every $i<\lambda$ let $S_i$ be a finite set of vertices that separates $\omega_i$ from $B$, chosen so that
  \begin{txteq}
    \label{eq:minimal}
    $S_i\setminus\bigcup_{j<i}S_j$ is minimal (under set containment).
  \end{txteq}
  In particular, if $\omega_i$ can be separated from $B$ by a finite subset of $\bigcup_{j<i}S_j$, then $S_i$ is such a subset.
  
  We claim that every set $S_i$ also satisfies
  \begin{txteq}
    \label{eq:ray}
    For every $s\in S_i\setminus\bigcup_{j<i}S_j$ and every finite $S\subset\bigcup_{j<i}S_j$, there exists a ray in $\omega_i$ that starts in $s$, avoids $S$, and is contained in $U_i\cup\{s\}$.
  \end{txteq}
  Indeed, for every $s\in S_i\setminus\bigcup_{j<i}S_j$ and every finite $S\subset\bigcup_{j<i}S_j$, the set $S'_i\assign S\cup S_i\setminus\{s\}$ does not separate $\omega_i$ from $B$, as this would contradict~\eqref{eq:minimal}. So there is a double ray $D$ that joins $\omega_i$ with an end in $B$ and avoids $S'_i$. As $S_i$ separates $\omega_i$ from $B$, $D$ hits $S_i$. But $D$ avoids $S_i\setminus\{s\}\subset S'_i$, so $D$ meets $S_i$ only in $s$. Thus, $D$ contains a ray as required in~\eqref{eq:ray}.
  
  Let us prove that $\quer{U}\cap B=\emptyset$. Suppose not, and pick $\omega\in\quer{U}\cap B$. As described earlier, there is a comb $C$ in $G$ with spine in $\omega$ and teeth in $\bigcup_{i<\lambda}S_i$. Let $Z$ be the set of its teeth. For every $z\in Z$ there is a smallest index $i=i(z)<\lambda$ with $z\in S_i$. Since the sets $S_i$ are finite, we may assume that $i(z)\not= i(z')$ for $z\not= z'$. Inductively, for all $j\in\NN$, choose $z(j)\in Z$ as the vertex $z\in Z\setminus\{z(k)\:|\:k<j\}$ with smallest value $i(z)$. Write $i(j)$ for $i(z(j))$. Note that the function $i(j)$ is strictly increasing. Hence for every positive integer $j$, the finite set $\bigcup_{k<j}S_{i(k)}$ is a subset of the (possibly infinite) set $\bigcup_{l<i(j)}S_l$.
  
  We now inductively define disjoint rays $R_j$ for all $j\in\NN$ such that $R_j\in\omega_j$ starts at $z(j)$. By the choice of $z(j)$ and the definition of $i(j)$, we have $z(j)\notin\bigcup_{l<i(j)}S_l$. In particular, as $z(j)\in S_{i(j)}$,
  \begin{equation}
    \label{eq:newsi}
    S_{i(j)}\not\subset\bigcup_{l<i(j)}S_l.
  \end{equation}
  By \eqref{eq:ray}, there exists a ray $R_j\in\omega_{i(j)}$ that starts in $z(j)$, avoids the finite subset $\bigcup_{k<j}S_{i(k)}$ of $\bigcup_{l<i(j)}S_l$ and is contained in $U_{i(j)}\cup\{z(j)\}$. As $R_j$ avoids $\bigcup_{k<j}S_{i(k)}$, we have for every $k<j$ either $R_j\subset U_{i(k)}$ or $R_j\cap U_{i(k)}=\emptyset$. If $R_j$ was contained in $U_{i(k)}$, then $\omega_{i(j)}$ would also be contained in $U_{i(k)}$. But then $\omega_{i(j)}$ could be separated from $B$ by the finite subset $S_{i(k)}$ of $\bigcup_{l<i(j)}S_l$. By~\eqref{eq:minimal}, this would imply $S_{i(j)}\subset\bigcup_{l<i(j)}S_l$, contradicting~\eqref{eq:newsi}. We thus have $R_j\cap U_{i(k)}=\emptyset$, as well as $z(k)\notin R_j$ for all $k<j$.
  
  Therefore, $\RRR\assign\{R_j\:|\: j<\aleph_0\}$ is a set of disjoint rays, where $R_j$ belongs to the end $\omega_{i(j)}$ and starts at the vertex $z(j)$. As every finite set of vertices misses both a tail of our comb $C$ and all but finitely many rays in $\RRR$, no finite set of vertices separates $\omega$ from $A$, in contradiction to the fact that $A$ is closed and $\omega\notin A$.
\end{proof}

\begin{proof}[Proof of Theorem~\ref{thm:Vhat}]
  Let $A,B$ be disjoint closed sets in $\hat V(G)$. As $A\cap\Omega(G)$ and $B\cap\Omega(G)$ are closed in $\hat V(G)$, Lemma~\ref{lemma:separatingends} gives us disjoint neighbourhoods $O$ of $A\cap\Omega(G)$ and $U$ of $B\cap\Omega(G)$ in $\hat V(G)$. Then
  \begin{align*}
    &(O\setminus B)\cup (A\cap V(G))&
    &\text{and}&
    &(U\setminus A)\cup (B\cap V(G))&
  \end{align*}
  are disjoint neighbourhoods of $A$ and $B$, respectively. Thus, $\hat V(G)$ is normal.
\end{proof}

\section{Topologies including edges}
\label{sec:incl_edges}

The topological space $\abs{G}$ of an infinite graph $G$ consists of the disjoint union of $V(G)$, $\Omega(G)$ and a copy $\offen{e}=(u,v)$ of $(0,1)$ for every edge $e=uv\in E(G)$. The bijection between $(0,1)$ and $(u,v)$ can be extended to a bijection of $[0,1]$ and $[u,v]\assign\{u\}\cup (u,v)\cup\{v\}$, which induces a metric on $[u,v]$. For any edge $e$, let $d_e$ denote this metric.

In~\cite{Die06},~\cite{DieGT},~\cite{DieKue04a},~\cite{DieKue04b} several topologies on $|G|$ are studied. We shall present one of them, called \textsc{MTop}. However, it turns out that all these topologies induce the same topology on $\hat V(G)$: the topology we defined in Section~\ref{sec:notation}.

\textsc{MTop} is generated by the following basic open sets. For every $z\in\offen{e}=(u,v)$ and $\eps$ such that $0<\eps\le\min\{d_e(u,z),d_e(v,z)\}$, we let the open $\eps$-ball around $z$ in $\offen{e}$ be open in $\abs{G}$ and denote it by $O_{\eps}(z)$. For every vertex $u$ and $\eps\in (0,1]$, we let the set of all points on edges $[u,v]$ of distance less than $\eps$ from $u$ (measured in $d_e$ for each $e=uv$) be open in $\abs{G}$ and denote it by $O_{\eps}(u)$. For every end $\omega$, $\eps\in (0,1]$ and every finite set $S$ of vertices, we let the set $\hat C_{\eps}(S,\omega)$ be open in $\abs{G}$, where $\hat C_{\eps}(S,\omega)$ consists of $\hat C(S,\omega)$, all inner points of edges that have both endvertices in $C(S,\omega)$, and, for each edge $uv$ from $C(S,\omega)$ to $S$, all points on $[u,v]$ of distance less than $\eps$ from $u$ (measured in $d_e$ for $e=uv$).

As a generalization of Theorem~\ref{thm:Vhat} (note that since $\hat V(G)$ is a closed subspace of $|G|$, normality of $|G|$ implies that $\hat V(G)$ is normal) we prove the following result.

\begin{theorem}
  \label{thm:graph}
  Let $G$ be an infinite graph. Then $|G|$ with \textsc{MTop} is normal.
\end{theorem}

\begin{proof}
  Let $A,B$ be disjoint closed sets in $\abs{G}$. As $A\cap\Omega(G)$ and $B\cap\Omega(G)$ are closed in $\abs{G}$, Lemma~\ref{lemma:separatingends} gives us disjoint neighbourhoods $\hat O$ of $A\cap\Omega(G)$ and $\hat U$ of $B\cap\Omega(G)$ in $\hat V(G)$. These sets can be extended to disjoint open sets in $|G|$: Indeed, adding all edges with both endvertices in $\hat O\cap V(G)$ as well as, for each edge $uv$ from $\hat O\cap V(G)$ to $V(G)\setminus\hat O$, all points on $[u,v]$ of distance less than $\frac12$ from $u$ yields a neighbourhood $O_1$ of $A\cap\Omega(G)$ in $|G|$. Likewise, we obtain a neighbourhood $U_1$ of $B\cap\Omega(G)$ in $|G|$ disjoint from $O_1$.

  We will now construct further neighbourhoods $O_2$ of $A\cap\Omega(G)$ and $U_2$ of $B\cap\Omega(G)$ as well as neighbourhoods $O_3$ of $A\setminus\Omega(G)$ and $U_3$ of $B\setminus\Omega(G)$ so that $O_3$ is disjoint from $U_2\cup U_3$ and $U_3$ is disjoint from $O_2\cup O_3$.

  Since $B$ is closed, there exists for every $a\in A$ a neighbourhood $\hat C_{\eps_a}(S_a,a)$ (if $a$ is an end) or $O_{\eps_a}(a)$ (if $a$ is a vertex or a point on an edge) of $a$ avoiding $B$. Choose $O_2$ as the union of all the open sets $\hat C_{\frac12\eps_a}(S_a,a)$ for $a\in A\cap\Omega(G)$ and $O_3$ as the union of all the open sets $O_{\frac12\eps_a}(a)$ for $a\in A\setminus\Omega(G)$. The sets $U_2$ and $U_3$ are chosen analogously.

  It is straightforward to check that these neighbourhoods satisfy the desired conditions. As $O_1\cap U_1=\emptyset$, we deduce that
  \begin{align*}
    &O\assign(O_1\cap O_2)\cup O_3&
    &\text{and}&
    &U\assign(U_1\cap U_2)\cup U_3&
  \end{align*}
  are disjoint neighbourhoods of $A$ and $B$, respectively. Thus, $\abs{G}$ is normal.
\end{proof}

\textsc{MTop} is the topology defined in~\cite{DieGT}. In~\cite{Die06},~\cite{DieKue04a},~\cite{DieKue04b} some more topologies on $\abs{G}$ are studied. These can equip $\abs{G}$ with certain desirable properties, such as metrizability, or compactness. (See~\cite{Die06} for characterizations of those graphs for which $|G|$ is metrizable or compact with this topologies.) Our proof of Theorem~\ref{thm:graph} can be adapted to one of those topologies, called \textsc{Top}. (The third, \textsc{VTop}, is not even Hausdorff.) We thus have

\begin{theorem}
  \label{thm:TOP}
  Let $G$ be an infinite graph. Then $|G|$ with \textsc{Top} is normal.\noproof
\end{theorem}

In some contexts, however, such as plane duality~\cite{BruDie05}, the most natural space associated with a graph $G$ is not $\abs{G}$, but a certain quotient space $\tilde G$ of $\abs{G}$~\cite{DieKue04b} (where $\abs{G}$ carries either \textsc{Top} or \textsc{VTop}, but one can also define $\tilde G$ starting from $|G|$ with \textsc{MTop}). In this section we show that $\tilde G$, and its end space $\tilde\Omega(G)$, are also normal. We may assume that the topology on $\abs{G}$ is \textsc{Top}, since we know that $\abs{G}$ is normal in this case.

To define $\tilde G$, let us say that a vertex $v$ \emph{dominates} an end $\omega$ if every finite set of vertices that separates $v$ from $\omega$ contains $v$. Let $\Omega_v$ denote the set of ends of $G$ that are dominated by the vertex $v$. Throughout this section, we assume that
\begin{equation} \label{eq:domin}\tag{\ensuremath{\ast}}
  \emtext{no end of $G$ is dominated by more than one vertex.}
\end{equation}
By~\eqref{eq:domin}, we have $\Omega_u\cap\Omega_v=\emptyset$ for all $u\not= v$. Let $\tilde G$ be the quotient space of $\abs{G}$ obtained by identifying each vertex $v$ with the ends in $\Omega_v$. This means that there is an identification map $\sigma:\abs{G}\twoheadrightarrow\tilde G$ with $\sigma(x)=v$ if $x\in\Omega_v$ and $\sigma(x)=x$ otherwise, and a subset $U$ of $\tilde G$ is open if and only if $\sigma^{-1}(U)$ is open in $\abs{G}$. We write $\tilde\Omega(G)$ for the set of undominated ends of $G$, which we informally also call the \emph{ends of} $\tilde G$. Note that $\tilde\Omega(G)$ is a subspace both of $\abs{G}$ and of $\tilde G$; the subspace topologies coincide (even if we had chosen the topology on $\abs{G}$ as \textsc{MTop} or \textsc{VTop}), and we endow $\tilde\Omega(G)$ with this topology.

If $G$ is connected, then by~\eqref{eq:domin} and Halin's~\cite{Hal78} theorem that connected graphs not containing a subdivision of an infinite complete graph have normal spanning trees, $\abs{G}$ is metrizable in \textsc{MTop}~\cite[Theorem~3.1(i)]{Die06}. Hence $\tilde\Omega(G)$, too, is a metric space, and therefore normal.

\begin{theorem}
  \label{thm:gtilde}
  For every graph $G$ satisfying~\eqref{eq:domin}, the space $\tilde G$ is normal.
\end{theorem}

Given a set $X\subset\abs{G}$, write $[X]$ for $\sigma^{-1}\big(\sigma(X)\big)$. (Thus, $[X]$ is the union of $X$ and all the sets of the form $\Omega_v\cup\{v\}$ that meet $X$.)

\begin{lemma}
  \label{lemma:quotient}
  If $X\subset\abs{G}$ is closed in $\abs{G}$, then so is $[X]$.
\end{lemma}
\begin{proof}
  If $[X]$ is not closed, there exists a point $x\notin[X]$ in the closure of $[X]$. Clearly, $x$ is an end.

  As $x$ does not lie in $X$, and $X$ is closed, there is an $\eps\in(0,1]$ and a finite set $S$ of vertices such that $\hat C_{\eps}(S,x)\cap X=\emptyset$. Then any point $z$ of $[X]$ in $\hat C_{\eps}(S,x)$ must lie in a set $\Omega_v\cup\{v\}$ that meets $X$. Since $S$ separates $x$---and therefore also $z$---from every point in $(\Omega_v\cup\{v\})\cap X$, the vertex $v$ has to lie in $S\cap[X]$ and $z$ is an end in $\Omega_v$.

  By a result of~\cite{DieKue04b}, the sets $\Omega_v$ are closed in $\abs{G}$, so the finite union $\bigcup_{v\in S\cap[X]}\Omega_v$ is also closed. The intersection of its complement in $\abs{G}$ with $\hat C_{\eps}(S,x)$ is a neighbourhood of $x$ that avoids $[X]$.
\end{proof}

\begin{proof}
[Proof of Theorem~\ref{thm:gtilde}]
  Let disjoint closed subsets $\tilde A,\tilde B$ of $\tilde G$ be given. Then $A\assign\sigma^{-1}(\tilde A)$ and $B\assign\sigma^{-1}(\tilde B)$ are disjoint closed sets in $\abs{G}$. By Theorem~\ref{thm:Vhat}, we have disjoint open sets $U\supset A$ and $V\supset B$ in $\abs{G}$.

  As $\abs{G}\setminus U$ is closed, Lemma~\ref{lemma:quotient} yields that $[\abs{G}\setminus U]$ is closed. Since $[A]=A$, by the definition of $A$, we have $A\cap[\abs{G}\setminus U]=\emptyset$. Hence, $U'\assign\abs{G}\setminus[\abs{G}\setminus U]$ is an open subset of $U$ that still contains $A$ and satisfies $[U']=U'$. Likewise, $V$ has an open subset $V'$ that still contains $B$ and satisfies $[V']=V'$.

  Thus, $\sigma(U')$ and $\sigma(V')$ are disjoint neighbourhoods of $\tilde A$ and $\tilde B$, respectively.
\end{proof}

\begin{noCounterPro}
  Is $\tilde G$ metrizable?
\end{noCounterPro}

Vella and Richter~\cite{VelRich06} solve this problem by proving that if $G$ is $2$-connected and no two vertices are connected by infinitely many internally disjoint paths, $\tilde G$ is even a Peano space.

\small
\vskip2mm plus 1fill
\parindent=0pt\obeylines

Philipp Spr\"ussel {\tt <philipp.spruessel@gmx.de>}
\smallskip
Mathematisches Seminar
Universit\"at Hamburg
Bundesstra\ss e 55
20146 Hamburg
Germany

\end{document}